\documentclass[12pt]{amsart}
\usepackage{amsmath}

\theoremstyle{plain}
\newtheorem{theorem}{Theorem}
\newtheorem{corollary}{Corollary}
\newtheorem{lemma}{Lemma}
\newtheorem{proposition}{Proposition}

\def\bcdot{\,\boldsymbol\cdot\,}
\def\lra{\longrightarrow}

\def\a{\alpha}
\def\cE{{\mathcal{E}}}
\def\si{\sigma}
\def\ph{\varphi}
\def\La{\Lambda}

\newcommand{\BL}{\biggl}
\newcommand{\BR}{\biggr}

\def\wt{\widetilde}
\def\wh{\widehat}

\let\rom\textup

\begin{document}

\author{M.~G.~Shur}

\title[Exponentials and $R$-recurrent random walks on groups]%
{Exponentials and $R$-recurrent random walks\\ on groups}

\address{Moscow Institute of Electronics and Mathematics,\newline
\indent National Research University Higher School of Economics}

\date{\today}



\begin{abstract}
On a locally compact group $E$ with countable base, we consider a
random walk $X$ that has a unique (up to a positive factor)
$r$-invariant measure for some $r>0$. Under some weak conditions on
the measure, there is a unique continuous exponential on $E$
naturally associated with $X$. It follows that there exists an
$R$-recurrent random walk in the sense of Tweedie on $E$ if and
only if $E$ is a recurrent group and there exists a Harris random
walk on~$E$.
\end{abstract}

\maketitle

\section{Introduction}\label{s1}

We deal with random walks on locally compact groups with countable
base. All groups considered below are assumed, without exception,
to have these properties, and the group operation is always written
as multiplication. The terminology pertaining to random walks and
irreducible Markov chains is mainly borrowed from the
books~\cite{8} and~\cite{4}, respectively.

Now let us start the exposition. It is well known that substantial
attention has recently been paid to the description of recurrent
groups, that is, groups on which there exists at least one
recurrent random walk. Most progress in this direction has been
made for abelian groups and connected Lie groups (see \cite{10},
\cite[Chap.~3]{8}, and~\cite{2}). In particular, it has turned out
that recurrent groups possess some special properties. (In
particular, they are unimodular~\cite{3}.)

In the present paper, we single out a fairly large family of random
walks that are not necessarily recurrent but can only be realized
on recurrent groups (see Theorem~\ref{th2} below). This family
consists of all possible spread out random walks that are
$R$-recurrent in the sense of Tweedie, or, which is the same, are
$\wt\pi$-irreducible $R$-recurrent Markov chains~\cite{4,5}, where
$\wt\pi$ is the Haar measure on the group and $R\ge1$ is the
convergence parameter of the Markov chain. Needless to say, such a
random walk is $\rho$-recurrent in the sense common in random walk
theory. (See~\cite[Chap~2]{6}; here $\rho$ is the spectral radius
of the random walk.) An arbitrary Bernoulli random walk on the
integer lattice generated by Bernoulli trials with success
probability $\rho\ne0,1$ can serve as an example of an
$R$-recurrent random walk.

The above-mentioned Theorem~\ref{th2} is preceded by
Theorem~\ref{th1} concerning conditions for the existence of a
unique continuous exponential associated with some random walk
(see~\eqref{eq1} below). In this connection, recall that a Borel
function $\ph>0$ defined on a group~$E$ is called an
\textit{exponential} on~$E$ if $\ph(xy)=\ph(x)\ph(y)$ for any
$x,y\in E$; such functions play a noticeable role in random walk
theory~\cite{6}. Theorem~\ref{th2} expresses the simple fact that
if a random walk is spread out and $R$-recurrent in the sense of
Tweedie, then the random walk on the same group corresponding to
the law $R\ph v$, where $v$ is the law of the original random walk,
is a Harris random walk in the standard sense~\cite{8,11}.

Theorem~\ref{th1} and~\ref{th2}, which are proved in Secs.~\ref{s2}
and~\ref{s4}, respectively, are the main results of the present
paper. The other assertions, which are mainly gathered in
Sec.~\ref{s3}, supplement Sec.~\ref{s1} and contain preliminary
material for Sec.~\ref{s4}.

In a subsequent publication, the author intends to use the theory
discussed here to further develop the results in~\cite{9}, in
particular, to obtain new strong ratio limit theorems.

Let us explain the main notation used in the paper. We everywhere
consider a group~$E$ of the type indicated above with the family
$\cE$ of Borel subsets. We fix a right Haar measure $\pi$ on $\cE$
and the corresponding left Haar measure $\pi_1$ such that
$\pi_1(A)=\pi(A^{-1})$ for any $A\in\cE$, where $A^{-1}=\{x\in
E\colon x^{-1}\in A\}$. The abbreviation ``a.e.'' stands for
``$\pi$-almost everywhere'' or ``$\pi$-almost every,'' depending on
the context.

We specify a random walk $X=(X_n;n\ge0)$ on $E$, which will be
subjected to various restrictions where necessary. We assign the
random walk $\wh X$ dual to $X$ to the law $\wh v$, where $\wh
v(A)=v(A^{-1})$, $A\in\cE$, and the transition operators
corresponding to $X$ and $\wh X$ are denoted by $P$ and $\wh P$,
respectively~\cite{8}. Thus,
\begin{equation*}
    Pf(x)=\int f(xy)v(dy),\qquad \wh Pf(x)=\int f(xy)\wh v(dy)
\end{equation*}
for all $x\in E$, where $f$ ranges over the family of all Borel
functions $f\colon E\lra(-\infty,\infty]$ bounded below.

Finally, if $\nu$ is a measure on $\cE$ and $f$ is a nonnegative
Borel function on $E$, then the measures $\nu P$ and $\mu=f\nu$ are
defined in the usual way,
\begin{equation*}
    \nu P(A)=\int p(x,A)\nu(dx),\qquad
    \mu(A)=\int_A f\,d\nu,\quad A\in\cE.
\end{equation*}

\section{Exponentials and random walks}\label{s2}

Let $r\in(0,\infty)$. A Borel function $f\colon E\lra[0,\infty]$ or
a measure $\nu$ defined on $\cE$ is said to be
$r$-\textit{invariant for a random walk}~$X$ if $\int f\,d\pi>0$,
$f\not\equiv\infty$, and $f=rPf$ or if $\nu(E)>0$ and $\nu=r\nu P$,
respectively~\cite{4}. (Many authors give a different
interpretation to similar notions; e.g., cf.~the definition of an
invariant (or, which is the same, harmonic) function
in~\cite{1,4,7}.)

An exponential $\ph$ defined on $E$ is an $r$-invariant function
for $X$ if and only if
\begin{equation}\label{eq1}
    r=\BL[\int\ph\,dv\BR]^{-1},
\end{equation}
because
\begin{equation*}
    P\ph(x)=\int\ph(x)\ph(y)v(dy)=\ph(x)\int\ph\,dv.
\end{equation*}
The same assertion holds for the case of~$\wh X$ except that
Eq.~\eqref{eq1} should be replaced with its counterpart that
contains the measure $\wh v$ instead of $v$ and is naturally called
the dual version of~\eqref{eq1}.

\begin{theorem}\label{th1}
Assume that, for some $r>0$, a random walk $X$ has a unique \rom(up
to a positive factor\rom) $r$-invariant measure $\pi_0$ that is
continuous with respect to $\pi$ and takes finite values on compact
subsets of~$E$. Then there exists a unique continuous exponential
$\ph$ on $E$ satisfying condition~\eqref{eq1}, and the function
$\psi=\ph^{-1}$ is the unique continuous exponential on $E$
satisfying the dual version of~\eqref{eq1}. The exponentials $\ph$
and $\psi$ are $r$-invariant for $X$ and $\wh X$, respectively, and
the measure $\pi_0$ coincides with $\psi\pi$ up to a positive
factor.
\end{theorem}

\begin{proof}
By the assumption of the theorem, $\pi_0=h_1\pi$ for some Borel
function locally $\pi$-integrable (i.e., $\pi$-integrable on any
compact subset of $E$). Let $v\ge0$ ($v\not\equiv0$) be a
continuous compactly supported function on~$E$. Then the function
\begin{equation*}
    h(x)=\int v(y)h_1(yx)\pi(dy)=\int v(yx^{-1})h_1(y)\pi(dy),
    \qquad x\in E,
\end{equation*}
takes finite values and cannot be zero on the entire $E$. Moreover,
it is continuous on $E$ by virtue of the last relation.

Let us verify that
\begin{equation}\label{eq2}
    \int f\,d\mu=r\int Pf\,d\mu
\end{equation}
for the measure $\mu=h\pi$ and for every Borel function $f\ge0$.
Indeed, if $\Delta$ is the modular function of $E$, then
\begin{equation*}
\begin{split}
    \int f(x)
      &=\Delta(y)\int f(y^{-1}x)h_1(x)\pi(dx)
      =\Delta(y)\int f(y^{-1}x)\pi_0(dx)
\\
      &=r\Delta(y)\int f(y^{-1}x)\pi_0P(dx)
      =r\Delta(y)\int Pf(y^{-1}x)\pi_0(dx)
\\
      &=r\int Pf(x)h_1(yx)\pi(dx)
\end{split}
\end{equation*}
for every $y\in E$ by \cite[Theorem~15.15]{7}, where the third
equality takes into account the $r$-invariance of $\pi_0$ and the
last equality uses the first three with $f$ replaced by~$Pf$.
Moreover,
\begin{equation}\label{eq3}
\begin{split}
    \int f\,d\mu
      &=\int f(x)\BL[\int v(y)h_1(yx)\pi(dy)\BR]\pi(dx)
      \\
      &=\int v(y)\BL[\int f(x)h_1(yx)\pi(dx)\BR]\pi(dy)
      \\
      &=r\int v(y)\BL[\int Pf(x)h_1(yx)\pi(dx)\BR]\pi(dy)
\end{split}
\end{equation}
according to the preceding computation. The first two equalities
in~\eqref{eq3} with $Pf$ substituted for $f$ show that the
right-hand side of~\eqref{eq2} is equal to the last expression
in~\eqref{eq3}. Hence \eqref{eq3} implies relation~\eqref{eq2},
which shows that the measure~$\mu$ is $R$-invariant.

Let us verify that
\begin{equation}\label{eq4}
    h(yx)=a(x)h(y)
\end{equation}
for any $x,y\in E$ but for now postpone the determination of
$a(y)>0$. To this end, we take a point $y\in E$ and a function $f$
of the same type as above and write out~\eqref{eq2} with $f(x)$
replaced by $f(y^{-1}x)$. As a result, we readily find that
\begin{equation*}
    \int f(x)h(yx)\pi(dx)=r\int Pf(x)h(yx)\pi(dx)
\end{equation*}
for a broad class of functions $f$, and hence the measure $\mu_y$,
where $\mu_y(dx)=h(yx)\pi(dx)$, is $r$-invariant for~$X$.
Consequently, $\mu_y$ coincides with $\pi_0$ and $\mu$ up to some
factors, so that $\mu_y=a(y)\mu$ for some $a(y)>0$. In other words,
for each $y\in E$, Eq.~\eqref{eq4} holds for a.e.~$X\in E$, and
since $h$ is continuous, it follows that Eq.~\eqref{eq4} holds for
any $x,y\in E$.

For $x=e$, where $e$ is the identity element of $E$, it follows
from~\eqref{eq4} that $h(y)=a(y)h(e)$, $y\in E$, and so $h(e)\ne0$
and $a(y)=h(y)/h(e)$. Accordingly, replacing the function $v$ used
in the definition of $h$ by the function $\a v$ with some $\a>0$ if
necessary, we can assume in what follows that $h(e)=1$ and
$a(y)=h(y)$ on $E$, and Eq.~\eqref{eq4} means in this case that $h$
is a continuous exponential on~$E$.

The function $\ph=h^{-1}$ is a continuous exponential as well, and
moreover, it satisfies condition~\eqref{eq1}. Indeed, by using the
random walk $\wh X$ dual to $X$, we can rewrite~\eqref{eq2} in the
form
\begin{equation}\label{eq5}
    \int fh\,d\pi=r\int f\wh Ph\,d\pi,
\end{equation}
where $f$ ranges over the same family of functions as above. We see
that $h=r\wh Ph$ and hence
\begin{equation}\label{eq6}
\begin{split}
    h(x)&=r\wh Ph(x)=r\int h(xy)v(dy)
        =rh(x)\int\frac{\wh v(dy)}{h(y^{-1})}
        \\
        &=rh(x)\int\frac{\wh v(dy)}{h(y^{-1})}
        =rh(x)\int\frac{v(dy)}{h(y)}
        =rh(x)\int \ph\,dv
\end{split}
\end{equation}
for a.e.~$x\in E$, because $h$ is an exponential and
$h(y^{-1})=h^{-1}(y)=\ph(y)$. By comparing the left- and rightmost
expressions in~\eqref{eq6}, we arrive at condition~\eqref{eq1} and
hence to the $r$-invariance of $\ph$ for~$X$.

We point out that all but the first equality in~\eqref{eq6} is
a~priori satisfied everywhere in $E$, and hence all terms of these
equalities, together with the last term, are continuous on~$E$.
Since $h$ is continuous, it follows that all equalities
in~\eqref{eq6} hold everywhere in $E$. The first of them
establishes the $r$-invariance of the exponential $h=\ph^{-1}=\psi$
for $\wh X$ and hence the validity of the dual version
of~\eqref{eq1} for~$\psi$.

Next, if some continuous exponential $\psi_1$ satisfies the same
version of~\eqref{eq1}, then $\psi_1=r\wh P\psi_1$, whence one
again obtains~\eqref{eq5} with $\psi_1$ substituted for~$h$. In
other words, we can apply~\eqref{eq2} with $\mu$ replaced by
$\mu_1=\psi_1\pi$, thus establishing the $r$-invariance of $\mu_1$
for $X$ together with the relation $\ph_1=k\psi$ for an appropriate
$k>0$. However, $\psi_1(e)=\psi(e)=1$; i.e., $k=1$ and
$\psi_1=\psi$. Thus, $\psi$ has the uniqueness property claimed in
the theorem.

One can readily justify a similar property of the exponential
$\phi$. Namely, if a continuous exponential $\phi_1$ satisfies
condition~\eqref{eq1}, then the dual version of~\eqref{eq1} applies
to $\psi_2=\ph_1^{-1}$, because
\begin{equation*}
    \int\psi_2\,d\wh
    v=\int\psi_2(x^{-1})v(dx)=\int\psi_1\,dv=r^{-1},
\end{equation*}
and by the preceding we have $\psi_2=\psi$ and
$\ph_1=\psi^{-1}=\ph$. The proof is complete.
\end{proof}

\begin{corollary}\label{cor1}
If, under the assumptions of Theorem~1, some function $g$ is
$r$-invariant for $\wh X$ and locally $\pi$-integrable, then
$g=k\ph$ a.e.\ for some $k>0$. If the assumptions of the theorem
hold for the random walk $\wh X$, then every function $g$ that is
$r$-invariant for $X$ and locally $\pi$-integrable has the form
$g=k_1\ph$ a.e.\ for some $k_1>0$.
\end{corollary}

To establish the first part of the corollary, it suffices to notice
that the measure $g\pi$ is in this case $r$-invariant for $X$ and
takes finite values on compact sets; hence it coincides, up to some
factors, with the measures $\pi_0$ and $\mu=\ph^{-1}\pi$ in the
proof of the theorem. The second part of the corollary follows from
the first.

There is a more convenient version of this corollary for spread out
$R$-recurrent random walks (see Proposition~\ref{p5}). As to the
condition that $r$-invariant functions be locally integrable, it is
often satisfied automatically (see Proposition~\ref{p2}).

\begin{corollary}\label{cor2}
Let the random walk $X$ be symmetric \rom(i.e., $v=\wh v$\rom), and
assume that, for some $r>0$, any of its $r$-invariant measures can
differ only by appropriate factors. Then $r=1$, and every invariant
\rom(i.e., $r$-invariant\rom) function for $X$ is a.e.\ equal to
some constant.
\end{corollary}

Indeed, in this case we can assume that $X=\wh X$, and
Theorem~\ref{th1} implies the identity $\ph\equiv\ph^{-1}$, by
which $\ph=1$ and $r=1$ (see~\eqref{eq1}). The remaining part of
Corollary~\ref{cor2} can be derived from Corollary~\ref{cor1}.

\section{Intermediate results}\label{s3}

In this section, on the one hand, we make some preparations for the
proof of Theorem~\ref{th2} in Sec.~\ref{s4}; on the other hand, we
give supplementary material related to Theorem~\ref{th1} and
Corollary~\ref{cor1} (see Proposition~\ref{p2}).

In all subsequent statements except for Proposition~\ref{p4}, $X$
is assumed to be a spread out random walk, that is, a random walk
for which some convolution power $v^n$, $n\ge1$, of the law $v$ is
nonsingular with respect to the Haar measure $\pi$. Moreover, as is
often done (cf.~\cite{8}), we restrict ourselves to adapted random
walks.

Let us start by discussing the irreducibility property, which is
interpreted differently in random walk theory and in the general
Markov chain theory. Namely, a random walk is said to be
\textit{irreducible} \cite{1,8} if the least closed semigroup
$T\subset E$ containing the support of the measure $v$ coincides
with $E$; it is said to be $\pi$-irreducible \cite{4} if
\begin{equation}\label{eq7}
    \sum_{n\ge1} p(n,x,A)>0
\end{equation}
for all $x\in E$ and $A\in\cE_+$, where $p(n,\bcdot,\bcdot)$
are the transition probabilities in $n$ steps corresponding to $X$
and $\cE_+=\{A\in\cE\colon\pi(A)>0\}$.

The proof of Proposition~\ref{p1}, which establishes that these
two definitions are essentially the same in the case of spread out
random walks, and of Propositions~\ref{p2} and~\ref{p3} is
based on Lemma~3.7 in~\cite[Chap.~3]{8}, which establishes the
existence of a compact set $V\subset E$ with nonempty interior, a
positive integer $m$, and positive numbers $a$ and $b$ such that
\begin{equation}\label{eq8}
 p(m,x,A)=v^m(x^{-1}A)\ge a\pi(V\cap x^{-1}A)
                      \ge b\pi_1(V\cap x^{-1}A)
\end{equation}
for any $x\in E$ and $A\in\cE$. (See the end of Sec.~\ref{s1} for
the definition of a left Haar measure.)

\begin{proposition}\label{p1}
A spread out random walk $X$ is $\pi$-irreducible if and only if it
is irreducible.
\end{proposition}

\begin{proof}
Assume that $X$ is irreducible and use relations~\eqref{eq8} for
$x\in E$, $A\in\cE_+$, and $V$, $m$, etc.\ chosen according to what
was just said. If $A\in\cE_+$ is relatively compact, then the last
expression in~\eqref{eq8} continuously depends on $x\in E$ and does
not vanish identically \cite[Corollary~20.17]{7}, so that
$p(m,x,A)>0$ for all $x$ in a nonempty open set. It follows from
the Chapman--Kolmogorov equation and the irreducibility of $X$ that
for each $x\in E$ one has $p(n,x,A)>0$ for some positive integer
$n=n(x)$, and this conclusion readily extends to arbitrary
$A\in\cE_+$. Thus, we have established the $\pi$-irreducibility
of~$X$.

Conversely, let $X$ be $\pi$-irreducible, and let $T$ be the
above-mentioned semigroup. The set $S=E\setminus T$ is open, and
hence $\pi(S)>0$ provided that $S$ is nonempty. But then
$\sum_{n\ge1} p(n,e,S)>0$ by~\eqref{eq7}, which contradicts the
definition of~$T$. Thus, $S$ is empty; i.e., the random walk $X$ is
irreducible.
\end{proof}

\begin{corollary}\label{cor3}
If a spread out random walk $X$ is irreducible, then each of the
Haar measures $\pi$ and $\pi_1$ is a maximal irreducibility measure
for $X$ in the sense of~\cite{4}.
\end{corollary}

This corollary, which is a straightforward consequence of
Proposition~\ref{p1}, permits freely using the theory developed
in~\cite{4} in the subsequent exposition.

Note also that, according to Proposition~\ref{p1}, we can replace
the condition of $\pi$-irreducibility in the definition of
$R$-recurrent random walk (see Sec.~\ref{s1}) by the condition of
irreducibility as long as we restrict ourselves to spread out
random walks.

Recall that the replacement of the equality $f=rPf$ in the
definition of $r$-invariant function (see Sec.~\ref{s1}) by the
inequality $f\ge rPf$ gives the definition of $r$-subinvariant
function~\cite{4}.

\begin{proposition}\label{p2}
If a random walk $X$ is spread out and irreducible, then, for each
$r>0$, every function $r$-subinvariant for $X$ is locally
$\pi$-integrable.
\end{proposition}

\begin{proof}
Assume the contrary: there exists an $r$-subinvariant function $f$
for $X$ that is not $\pi$-integrable on a nonempty compact set~$F$.
Then $f$ is not $\pi$-integrable and hence not $\pi_1$-integrable
in any neighborhood of some point $z\in F$. (Otherwise, it would be
$\pi$-integrable in some neighborhood $G_x$ of each point $x\in F$,
and finitely many sets of the form $G_x$, $x\in F$, would form an
open cover of the compact set $F$, which is only possible if $f$ is
$\pi$-integrable on $F$.)

Fix a point $z$ with this property and again use
relation~\eqref{eq8} retaining the preceding notation. Since the
group $E$ is regular \cite[Chap.~1]{7}, it follows that there
exists a nonempty open set $W$ whose closure is contained in the
set $V$ indicated in~\eqref{eq8} and a neighborhood $W_0$ of the
identity element $e$ such that $x^{-1}W\subset V$ for all $x\in
W_0$. For any $x\in W_0$ and any Borel set $A\subset W$, the last
expression in~\eqref{eq8} coincides with
\begin{equation}\label{eq9}
    b\pi_1(x^{-1}A)=b\pi_1(A),
\end{equation}
because $\pi_1$ is a left Haar measure.

Now take a $y\in W$, set $s=yz^{-1}$ and $g(x)=f(s^{-1}x)$, $x\in
E$, and note that
\begin{multline*}
    \int_Wg\,d\pi_1=\int 1_W(x)f(s^{-1}x)\pi_1(dx)\\
                   =\int 1_W(sx)f(x)\pi_1(dx)
                   =\int_{s^{-1}W}f\,d\pi_1=\infty
\end{multline*}
(where $1_W$ is the indicator function of the set $W$), because
$s^{-1}W$ is a neighborhood of the point $s^{-1}y=z_0$ and $f$ is
nonintegrable in any neighborhood of that point. Hence we have,
by~\eqref{eq8} and~\eqref{eq9},
\begin{equation*}
    P^mg(x)\ge b\int_Wg\,d\pi_1=\infty,\quad x\in W_0,
\end{equation*}
and consequently, $P^mf(x)=P^mg(sx)=\infty$ if $x$ ranges over the
open set $s^{-1}W_0$. Clearly, $f\equiv\infty$ on the same set by
virtue of the $r$-subinvariance of $f$; in conjunction with the
Chapman--Kolmogorov equation, this implies that $f\equiv\infty$ on
$E$ (cf.~the preceding proof), which is inconsistent with the
definition of $r$-subinvariant function. Thus, our assumption at
the beginning of the proof is wrong, and the proof is complete.
\end{proof}

\begin{proposition}\label{p3}
If a random walk $X$ is spread out and irreducible, then there
exists an open set $U\subset E$ such that the set $sU$ and each of
the measures $1_{sU}\pi_1$ and $1_{sU}\pi$ is small for $X$ in the
sense of~\cite{4}.
\end{proposition}

\begin{proof}
Relations~\eqref{eq8} and~\eqref{eq9} imply the inequality
\begin{equation}\label{eq10}
    p(m,x,A)\ge b\pi_1(A),
\end{equation}
where $x$ ranges over some neighborhood $W_0$ of $e$, the set $A$
is an arbitrary Borel subset of a nonempty open set $W\subset E$,
and a positive integer $m$ and a $b>0$ are chosen appropriately.
Furthermore, without loss of generality we can assume that $W_0$ is
relatively compact. According to \cite[Definition~2.3]{4},
inequality~\eqref{eq10} shows that the function $1_{W_0}$ and the
measure $1_{W_0}\pi_1$ are small for $X$. Now take an $s\in E$. If
$x\in sW_0$ and $A\subset sW$ ($A\in\cE$), then, by~\eqref{eq10},
\begin{equation*}
     p(m,x,A)=p(m,s^{-1}x,s^{-1}A)\ge b\pi_1(s^{-1}A)=b\pi_1(A),
\end{equation*}
because $s^{-1}x\in W_0$ and $s^{-1}A\subset W$. Thus, we again
arrive at~\eqref{eq10} but with somewhat different $x$ and $A$, and
this time the set $sW_0$ and the measure $1_{sW}\pi_1$ prove to be
small. In view of the relative compactness of $W$ and the relation
between the right and left Haar measures \cite[Sec.~15]{7}, we see
that the measure $1_{sW}\pi$ is small as well. To complete the
proof, it remains to set $U=W_0\cap s_0W$ for an $s_0\in E$ such
that $U$ is nonempty.
\end{proof}

In the following section, we need an intuitively clear statement
(see Proposition~\ref{p4}) which will help us establish that a
certain random walk is a Harris random walk. The proof of this
statement is based on the following lemma, where the symbol $M_x$,
$x\in E$, stands for the expectation corresponding to the
probability measure $P_x$ (which is defined on the corresponding
$\si$-algebra of events and is assigned to $X$ for the initial
state $x$ \cite{4,8}).

\begin{lemma}\label{l1}
For every bounded Borel function $g$ on $E$, the relation
\begin{equation}\label{eq11}
    M_{yx}[g(X_n)]=M_{x}[g(yX_n)],\qquad n\ge0,
\end{equation}
holds for any $x,y\in E$.
\end{lemma}

\begin{proof}
For $n=0$, this is obvious, so consider the case of $n\ge1$. If
$g=1_A$ and $A\in\cE$, then the left- and right-hand sides
of~\eqref{eq11} are equal to the probabilities $P_{yx}(X_n\in A)$
and $P_{x}(X_n\in y^{-1}A)$, respectively, whence~\eqref{eq11}
follows in our case. Obviously, Eq.~\eqref{eq11} remains valid if
$g$ is a Borel function taking finitely many values. If this
condition is violated, then $g$ can be uniformly approximated by
functions satisfying this condition, and we again arrive
at~\eqref{eq11}.
\end{proof}

Consider the functions $h^B(x)=P_x(\La_B^1)$ and
$H^B(x)=P_x(\La_B)$, where $\La_B^1$ and $\La_B$ are the events
that the trajectory of the random walk $X$ hits a set $B\in\cE$ at
least once or infinitely many times, respectively.

\begin{proposition}\label{p4}
One has
\begin{equation}\label{eq12}
    h^{yB}(yx)=h^b(x),\qquad
    H^{yB}(yx)=H^b(x)
\end{equation}
for $x,y\in E$, where $yB=\{z=yx\colon x\in B\}$.
\end{proposition}

\begin{proof}
Most of the proof deals with the verification of the first relation
in~\eqref{eq12}. First, let us verify that
\begin{equation}\label{eq13}
    h_n^{yB}(yx)=h_n^B(x),\qquad n\ge0,
\end{equation}
for $x,y\in E$, where $h_0^B(x0)=P_x(X_0\in B)$ and
\begin{equation*}
    h_n^B(x)=P_x(X_0\notin B, X_1\notin B,\dotsc, X_{n-1}\notin B,
    X_n\in B),\qquad n\ge1.
\end{equation*}
Relation~\eqref{eq13} is obvious for $n=0$: both parts are
simultaneously $1$ or $0$ depending on whether $x\in B$ or $x\notin
B$, respectively. The case of $n\ge1$ and $x\in B$ is equally easy,
so we assume from now on that $x\notin B$. Assume that \eqref{eq13}
has been proved for some $n\ge0$. Then, by the Markov property,
\begin{equation}\label{eq14}
    h_{n+1}^{yB}(yx)=M_{yx}[1_\La h_n^{yB}(X_1)]
\end{equation}
with the factor $1_\La$ that is the indicator function of the event
$\La=\{X_0\notin B\}$. Since $x\notin B$, we have
$P_{yx}(\La)=P(yX_0\notin yB)=1$, and hence the right-hand side
of~\eqref{eq14} coincides with
\begin{equation*}
    M_{yx}[h_n^{yB}(X_1)]=M_{x}[h_n^{yB}(yX_1)]
    =M_{x}[h_n^{B}(X_1)]=h_{n+1}^{B}(x).
\end{equation*}
(Here the first equality follows from~\eqref{eq11}; the second,
from~\eqref{eq13} with the current value of $n$; and the third,
from~\eqref{eq14} with $y=e$.) Thus, we have proved~\eqref{eq13}
with $n$ replaced by $n+1$, and so \eqref{eq13} holds for all
$n\ge0$.

Since $h^B(x)=\sum_{n\ge0} h_n^B(x)$, we have simultaneously proved
the first relation in~\eqref{eq12}, which, in conjunction
with~\eqref{eq11}, implies the relations
\begin{equation}\label{eq15}
    M_{yx}[h^{yB}(X_n)]=M_{x}[h^{yB}(yX_n)]=M_{x}[h^{B}(X_n)].
\end{equation}
Thus, the left- and right-hand sides of~\eqref{eq15} coincide, and to
justify the second relation in~\eqref{eq12}, it remains to note
that $M_{x}[h^{B}(X_n)]\to H^B(x)$ as $n\to\infty$ $(x\in E)$.
\end{proof}

\section{Random walks $R$-recurrent in the sense of Tweedie}\label{s4}

The main goal of this section is to prove the second main result of
this paper, Theorem~\ref{th2}. However, first we show how
dramatically Theorem~\ref{th1} and Corollary~\ref{cor1} are
simplified if the random walk is spread out and $R$-recurrent in
the sense of Tweedie, where, as before, $R$ stands for the
convergence parameter of~$X$.

\begin{proposition}\label{p5}
Under the above-mentioned conditions,
\begin{enumerate}
    \item[\rom{(i)}] There exists a unique continuous exponential
    $\ph$ on $E$ satisfying condition~\eqref{eq1} with $r=R$.
    \item[\rom{(ii)}] Every function $g$ that is $R$-invariant for
    $X$ or $\wh X$ has the form indicated in the first or second
    assertion, respectively, of Corollary~\rom{\ref{cor1}}.
    Moreover, functions $R$-subinvariant for $X$ or $\wh X$ have
    the same form.
\end{enumerate}
\end{proposition}

\begin{proof}
One can readily establish that $\wh X$ is spread out, irreducible,
and hence $\pi$-irreducible. Let us prove that this random walk is
$R$-recurrent in the sense of Tweedie. For a Borel function
$f\colon E\lra[0,\infty)$, set $G^Rf=\sum_{n\ge0}R^nP^nf$. We
define $\wh G^Rf$ with the use of the operator $\wh P$ in a similar
way. Assume that $\int f\,d\pi>0$ and a function $g$ has the same
properties as $f$. Since $G^Rf\equiv\infty$ owing to the
$R$-recurrence of $X$, it follows that
\begin{equation}\label{eq16}
    \int f\wh G^Rg\,d\pi=\int gG^Rf\,d\pi=\infty,
\end{equation}
and by setting $\nu=f\pi$, we obtain $\int \wh G^Rg\,d\nu=0$. By
Proposition~\ref{p3}, we can subject the function $f$ to the
requirement that the measure $\nu$ thus introduced be small for
$\wh X$. Now if $\wh G^Rg\not\equiv\infty$, then the function $\wh
G^Rg$ is $R$-subinvariant for $\wh X$, and the last equality
contradicts Proposition~5.1 in~\cite{4} provided that $g$ is small for
$\wh X$. Consequently, $\wh G^Rg\equiv\infty$ for all $g$ small for
$\wh X$; i.e., $\wh R\le R$, where $\wh R$ is the convergence
parameter of the random walk $\wh X$. By interchanging $X$ and $\wh
X$, we obtain the inequality $R\le\wh R$, and hence $R=\wh R$.

As a result, the random walk $\wh X$ proves to be $R$-recurrent as
well, which implies that Proposition~\ref{p5} holds. (See
\cite[Theorem~5.3]{4} as well as Theorem~\ref{th1} and
Corollary~\ref{cor1} in Sec.~\ref{s2}.)
\end{proof}

Let us proceed to the main goal of this section. Starting from the
law $v$ of a Tweedie $R$-recurrent random walk $X$ and the
exponential $\ph$ mentioned in Proposition~\ref{p5}\,(ii), consider
a random walk $\wt X$ on $E$ with the law $\wt v=R\ph v$ and the
transition operator~$\wh P$. (The measure $\wt v$ is a probability
measure by virtue of the $R$-invariance of $\ph$ for $X$ and
condition~\eqref{eq1} with $r=R$.) The convolution powers of the
new law can readily be expressed via the same powers of~$v$,
\begin{equation}\label{eq17}
    \wt v^n=R^n\ph v^n,\qquad n\ge1.
\end{equation}
For example, for each Borel function $f\colon E\lra[0,\infty)$ we
find that
\begin{equation*}
    \int f\,d(\wt v^2)=R^2\int f(xy)\ph(x)\ph(y)v(dx)v(dy)
                      =R^2\int f\ph\,d(v^2),
\end{equation*}
which gives~\eqref{eq17} for $n=2$. (By induction, one can
prove~\eqref{eq17} with the use of similar computations for all
$n\ge1$.) In turn, \eqref{eq17} readily implies that
\begin{equation*}
    \wt P^nf=R^n\frac1\ph P^n(f\ph),\qquad n\ge1,
\end{equation*}
which means that $\wt P$ can be obtained from $P$ by a so-called
similarity transformation~\cite{4}. (One also says that $\wt X$ is
obtained from $X$ by a passage to a $\ph$-process~\cite{6}.)

\begin{theorem}\label{th2}
The random walk $\wt X$ is a Harris random walk~\rom{\cite{4}}, and
hence the group $E$ is recurrent.
\end{theorem}

\begin{proof}
By \cite[Proposition~5.3]{4}, $\wt X$ is $\pi$-irreducible and
recurrent in the sense that $\wt H^B>0$ everywhere in $E$ and $\wt
H^B=1$ a.e.\ in $E$ whenever $B\in\cE_+$, where $\wt H^B$ is the
counterpart of the function $H^B$ (Proposition~\ref{p5}) for~$\wt
X$. By~\cite[Proposition~3.13]{4}, the recurrence of $\wt X$
implies the existence of a Harris set $E_1\in\cE_+$, i.e., a set
such that the restriction of $\wt X$ to $E_1$ is a Harris recurrent
Markov chain.

Take an $x\in E$ and a $B\in\cE_+$. If $z\in E_1$ and $y=zx^{-1}$,
then, by Proposition~\ref{p5},
\begin{equation*}
    \wt H^B(x)=\wt H^{yB}(yx)=\wt H^{yB}(z),
\end{equation*}
and since $z\in E_1$ and $\pi(yB)>0$, we have $\wt H^{yB}(z)=1$. In
other words, $\wt H^B(x)=1$ for all $x\in E$, and hence the random
walk $\wt X$ is a Harris random walk \cite[Definition~3.5]{4}. The
proof of the theorem is complete.
\end{proof}

\providecommand{\bysame}{\leavevmode\hbox to3em{\hrulefill}\thinspace}

\end{document}